\documentclass[a4paper,11pt]{article}
\usepackage{amsmath, amsfonts, amscd, amssymb, amsthm, enumerate, xypic}

\def\Mat{\text{M}}
\def\GL{\text{GL}}

\newcommand{\Ker}{\operatorname{Ker}}
\newcommand{\Vect}{\operatorname{Span}}
\newcommand{\im}{\operatorname{Im}}
\newcommand{\tr}{\operatorname{tr}}
\newcommand{\rk}{\operatorname{rk}}

\renewcommand{\setminus}{\smallsetminus}

\def\K{\mathbb{K}}

\def\N{\mathbb{N}}

\renewcommand{\L}{\mathbb{L}}

\def\calA{\mathcal{A}}

\def\calC{\mathcal{C}}

\def\calF{\mathcal{F}}

\def\calH{\mathcal{H}}

\def\calP{\mathcal{P}}

\def\lcro{\mathopen{[\![}}
\def\rcro{\mathclose{]\!]}}

\theoremstyle{definition}
\newtheorem{Def}{Definition}

\theoremstyle{plain}
\newtheorem{theo}{Theorem}
\newtheorem{prop}[theo]{Proposition}
\newtheorem{cor}[theo]{Corollary}
\newtheorem{lemme}[theo]{Lemma}

\theoremstyle{plain}

\theoremstyle{remark}
\newtheorem{Rems}{Remarks}
\newtheorem{Rem}[Rems]{Remark}

\title{On small matrix subalgebras with a trivial centralizer
\footnote{E-mail address: dsp.prof@gmail.com}}
\author{Cl\'ement de Seguins Pazzis \\
\begin{footnotesize}
\emph{Lyc\'ee Priv\'e Sainte-Genevi\`eve, 2, rue
de l'\'Ecole des Postes, 78029 Versailles Cedex, FRANCE.}
\end{footnotesize}}

\date{}

\begin{document}

\maketitle

\begin{abstract}
Given an integer $n \geq 3$, we investigate the minimal dimension of a subalgebra of $\Mat_n(\K)$
with a trivial centralizer. It is shown that this dimension is $5$ when $n$ is even
and $4$ when it is odd. In the latter case, we also determine all $4$-dimensional subalgebras with a trivial centralizer.
\end{abstract}



\renewcommand{\labelitemi}{$\bullet$}

\section{Introduction}

Here, we set an integer $n \geq 2$ and a field $\K$.
Using the French convention, we let $\N$ denote the set of non-negative integers and $\N^*$ the one of
positive integers. We let $\Mat_n(\K)$ denote the algebra of square matrices of order $n$ with entries in $\K$
and $T_n(\K)$ its subalgebra of upper triangular matrices. We let $\Mat_{p,q}(\K)$
denote the set of matrices with $p$ rows, $q$ columns and entries in $\K$. \\
All subalgebras of $\Mat_n(\K)$ are required to contain the unit matrix $I_n$:
the subalgebra $\Vect(I_n)$ will be called \textbf{trivial}. \\
For $(i,j)\in \lcro 1,n\rcro^2$, we let $E_{i,j}$ denote the elementary matrix of $\Mat_n(\K)$
with a zero entry in every position except for $(i,j)$ where the entry is $1$. \\
Given a subset $V$ of $\Mat_n(\K)$, we let
$$\calC(V):=\bigl\{A \in \Mat_n(\K) : \; \forall M \in V, \; AM=MA\bigr\}$$
denote its \textbf{centralizer}, and we simply write $\calC(A)=\calC(\{A\})$
when $A$ is a matrix of $\Mat_n(\K)$.
Recall that $\calC(V)$ is always a subalgebra of $\Mat_n(\K)$ and that
$\calC(V)$ is also the centralizer of the subalgebra generated by $V$. \\
The Jordan matrix of order $n$ will be written $J_n=(\delta_{i+1\,,\,j})_{1 \leq i,j \leq n}$,
where $\delta_{a,b}$ equals $1$ if $a=b$, and $0$ otherwise.

\vskip 3mm
\noindent
In this paper, we will focus on subalgebras of $\Mat_n(\K)$ with a small dimension and a trivial centralizer.
The basic motivation for studying subalgebras with a trivial centralizer comes from the theory of representations of
algebras. Let $\calA$ be a subalgebra of $\Mat_n(\K)$, which we identify with the algebra of linear endomorphisms of
the vector space $\K^n$. This yields a structure of $\calA$-module on $\K^n$ for which the endomorphisms
naturally correspond to the matrices in the centralizer $\calC(\calA)$.
When $\K$ is algebraically closed and $\K^n$ is a simple $\calA$-module (i.e.\ when it has no non-trivial submodule),
then $\calC(\calA)$ is trivial, however the converse may not hold. In the case $\calA$ is generated by a finite subgroup
of $\GL_n(\K)$ and $\K$ has characteristic $0$, then the converse is classically true because $\calA$ is then semi-simple
(see the theory of linear representations of finite groups).
In the general case of an arbitrary field and an arbitrary subalgebra of $\Mat_n(\K)$, the condition
$\calC(\calA)=\Vect(I_n)$ may thus be seen as an alternative notion of simplicity or semi-simplicity,
which provides motivation enough for studying it systematically.

\vskip 3mm
\noindent
Non-trivial subalgebras of $\Mat_n(\K)$ with a trivial centralizer are actually commonplace.
A classical example is that of $T_n(\K)$. Indeed, let $A \in \calC\bigl(T_n(\K)\bigr)$. \\
Then $A$ commutes with $E_{i,i}$ for every $i \in \lcro 1,n\rcro$, which shows that $A$ is diagonal. \\
The commutation of $A$ with $E_{1,i}$ for every $i \in \lcro 2,n\rcro$ then shows that all diagonal entries of $A$ are equal.

It is somewhat harder to produce such subalgebras with a small dimension. Our main goal
is to find the smallest dimension for such a subalgebra and to classify the subalgebras
of minimal dimension.

\begin{Def}
We let $t_n(\K)$ (or simply $t_n$ when the field is obvious)
denote the smallest dimension of a subalgebra of $\Mat_n(\K)$ with a trivial centralizer.
\end{Def}

Notice first that a subalgebra of dimension $2$
is always of the form $\K[A]$ for some matrix $A$ which is not a scalar multiple of $I_n$, so its centralizer contains $A$
and is therefore non-trivial. This essentially solves the case $n=2$.

\begin{prop}
One has $t_2=3$. More precisely, $T_2(\K)$ has a trivial centralizer and dimension $3$.
\end{prop}

\noindent We will assume $n \geq 3$ from now on. Our main results are stated below:

\begin{theo}\label{maintheo}
If $n \geq 3$ is even, then $t_n(\K)=5$. \\
If $n \geq 3$ is odd, then $t_n(\K)=4$.
\end{theo}

\begin{prop}\label{existdimpaire}
Let $p \in \N \setminus \{0,1\}$ and consider the linear subspace
$\calF_{2p}$ of $\Mat_{2p}(\K)$ generated by the matrices
$$\begin{bmatrix}
I_p & 0 \\
0 & 0
\end{bmatrix},\begin{bmatrix}
0 & 0 \\
0 & I_p
\end{bmatrix},\begin{bmatrix}
0 & I_p \\
0 & 0
\end{bmatrix},\begin{bmatrix}
0 & J_p \\
0 & 0
\end{bmatrix}
 \; \textrm{and} \; \begin{bmatrix}
0 & J_p^t \\
0 & 0
\end{bmatrix}.$$
Then $\calF_{2p}$ is a $5$-dimensional subalgebra of $\Mat_{2p}(\K)$ with a trivial centralizer.
\end{prop}

\begin{prop}\label{existdimimpaire}
Let $p \in \N^*$.
Consider the matrices
$C_p=\begin{bmatrix}
I_p & 0
\end{bmatrix}$ and
$D_p=\begin{bmatrix}
0 & I_p
\end{bmatrix}$ in $\Mat_{p,p+1}(\K)$, and define $\calH_{2p+1}$
as the linear subspace of $\Mat_{2p+1}(\K)$ generated by the matrices
$$\begin{bmatrix}
I_p & 0 \\
0 & 0
\end{bmatrix},
\begin{bmatrix}
0 & 0  \\
0 & I_{p+1}
\end{bmatrix},
\begin{bmatrix}
0 & C_p  \\
0 & 0
\end{bmatrix}
\;\textrm{and} \;
\begin{bmatrix}
0 & D_p  \\
0 & 0
\end{bmatrix}.$$
Then $\calH_{2p+1}$ is a $4$-dimensional subalgebra of $\Mat_{2p+1}(\K)$ with a trivial centralizer.
\end{prop}

\noindent Finally, we will prove that the latter example is essentially unique:

\begin{prop}\label{unicity}
Let $p \in \N^*$ and $\calA$ be a subalgebra of $\Mat_{2p+1}(\K)$ of dimension $4$ with a trivial centralizer.
Then $\calA$ is \textbf{conjugate} to either $\calH_{2p+1}$ or its transposed subalgebra $\calH_{2p+1}^t$,
i.e.\ there is a non-singular matrix $P \in \GL_{2p+1}(\K)$ such that
$$\calA=P\,\calH_{2p+1}\,P^{-1} \quad \text{or} \quad \calA=P\,\calH_{2p+1}^t\,P^{-1}.$$
\end{prop}

\begin{Rem}
It is an easy exercise to prove that $\calH_{2p+1}$ and $\calH_{2p+1}^t$ are not conjugate one to the other.
\end{Rem}

\section{Checking the examples}\label{proofexamples}

We will start with a little lemma.

\begin{lemme}\label{commutantsubspace}
Let $n \geq 2$ be an integer. Then
$\Vect(J_n,J_n^t)$ has a trivial centralizer.
\end{lemme}

\begin{proof}
Since $J_n$ is cyclic, its centralizer is $\K[J_n]$ (see \cite{Comm} Theorem 5 p.23),
and it thus contains only upper triangular matrices with equal diagonal entries.
Similarly, every matrix of $\K[J_n^t]$ is lower triangular.
It follows that every matrix in the centralizer of $\Vect(J_n,J_n^t)$
must be scalar.
\end{proof}

The examples featured in Propositions \ref{existdimpaire} and \ref{existdimimpaire} are based upon the
same idea. Consider a decomposition $n=p+q$ and a linear subspace $V$ of
$\Mat_{p,q}(\K)$. It is then easily checked that
$$\calH:=\Biggl\{
\begin{bmatrix}
a.I_p & K \\
0 & b.I_q
\end{bmatrix} \mid (a,b)\in \K^2, \; K \in V\Biggr\}$$
is always a subalgebra of $\Mat_n(\K)$ with dimension $\dim V+2$.
Straightforward computation also shows that the centralizer of the matrix
$P=\begin{bmatrix}
I_p & 0 \\
0 & 0
\end{bmatrix}$ is
$$\calC(P)=\Biggl\{\begin{bmatrix}
X & 0 \\
0 & Y
\end{bmatrix} \mid (X,Y)\in \Mat_p(\K) \times \Mat_q(\K)\Biggr\}.$$
Let $M=\begin{bmatrix}
X & 0 \\
0 & Y
\end{bmatrix} \in \calC(P)$. Then $M$ belongs to $\calC(\calH)$ if and only if
$\forall K \in V, \; X\,K=K\,Y$. Of course, this last relation need only be tested on
a basis of $V$.

\noindent From there, our claims may easily be proven.

\vskip 2mm
\noindent
\textbf{The example in Proposition \ref{existdimpaire}.} \\
Here, $q=p$ and $V=\Vect(I_p,J_p,J_p^t)$. \\
Let $(X,Y)\in \Mat_p(\K)^2$ such that $XI_p=I_pY$, $XJ_p=J_pY$ and $X\,J_p^t=J_p^t\,Y$.
Then $Y=X$ and $X$ commutes with both $J_p$ and $J_p^t$. By Lemma \ref{commutantsubspace}, $X$ is a scalar multiple of $I_p$ hence
$\begin{bmatrix}
X & 0 \\
0 & Y
\end{bmatrix}$ is a scalar multiple of $I_n$. This proves Proposition \ref{existdimpaire}.

\vskip 2mm
\noindent
\textbf{The example in Proposition \ref{existdimimpaire}.} \\
Here $q=p+1$ and $V=\Vect(C_p,D_p)$. \\
Let $(X,Y)\in \Mat_p(\K) \times \Mat_{p+1}(\K)$ such that
$\begin{bmatrix}
X & 0 \\
0 & Y
\end{bmatrix} \in \calC(\calH)$. \\
The identity $XC_p=C_pY$ entails that
$Y=\begin{bmatrix}
X & 0 \\
L & \alpha
\end{bmatrix}$ for some $(\alpha,L)\in \K \times \Mat_{1,p}(\K)$, whilst
identity $XD_p=D_pY$ shows that
$Y=\begin{bmatrix}
\beta & L' \\
0 & X
\end{bmatrix}$ for some $(\beta,L')\in \K \times \Mat_{1,p}(\K)$. We thus have
$$\begin{bmatrix}
X & 0 \\
L & \alpha
\end{bmatrix}=\begin{bmatrix}
\beta & L' \\
0 & X
\end{bmatrix}.$$
  Starting from the first column of $X$, an easy induction shows that $X$ is upper triangular with all diagonal entries equal to $\beta$. Also, starting from the last column of $X$, an easy induction shows that $L'=0$ and $X$ is lower triangular.
This yields $X=\beta.I_p$ and $Y=\beta.I_{p+1}$ hence $\begin{bmatrix}
X & 0 \\
0 & Y
\end{bmatrix}=\beta.I_{2p+1}$, which proves Proposition \ref{existdimimpaire}.

\section{A lower bound for $t_n(\K)$}

\subsection{Introduction}

Here, we will prove that $t_n \geq 4$ when $n$ is odd, and $t_n \geq 5$ when $n$ is even.
In other words, we will prove that every subalgebra of $\Mat_n(\K)$ has a non-trivial centralizer
provided it has dimension $p \leq 3$ when $n$ is odd, and dimension $p \leq 4$ when $p$ is even.
The proof is essentially laid out as follows:
\begin{itemize}
\item by extending the ground field, we reduce the study to the case of an algebraically closed field;
\item in this case, we discard all \emph{unispectral} subalgebras (i.e.\ subalgebras in which every operator has a sole eigenvalue);
\item in the remaining cases, the considered subalgebra contains a non-trivial idempotent which we use
to split the algebra $\calA$ into several remarkable subspaces; we then use that splitting to find a non-scalar matrix in the centralizer of $\calA$ when $\dim \calA$ is small enough.
\end{itemize}

\noindent From now on, we set an integer $n \geq 3$ and a subalgebra $\calA$ of $\Mat_n(\K)$.
The following elementary facts will be used repeatedly:
\begin{itemize}
\item for every $P \in \GL_n(\K)$, the conjugate subalgebra $P\,\calA\,P^{-1}$
has the same dimension as $\calA$ and its centralizer $P\,\calC(\calA)\,P^{-1}$ is trivial if and only if
$\calC(\calA)$ is trivial.
\item the transposed subalgebra $\calA^t:=\bigl\{M^t \mid M \in \calA\bigr\}$
has the same dimension as $\calA$ and its centralizer $C(\calA)^t$ is trivial
if and only if $C(\calA)$ is trivial.
\end{itemize}

\subsection{Reduction to the case of an algebraically closed field}\label{red}

Let $\L$ be a field extension of $\K$.
Recall that when $\calA$ is a subalgebra of $\Mat_n(\K)$
and we let $\calA_\L$ denote the linear $\L$-subspace of $\Mat_n(\L)$ generated by
$\calA$, then $\calA_\L$ is a subalgebra of $\Mat_n(\L)$.
The natural isomorphism of $\L$-algebras $\Mat_n(\L) \simeq \Mat_n(\K) \otimes_\K \L$
maps $\calA_\L$ to $\calA \otimes_\K \L$ hence $\dim_\L \calA_\L=\dim_\K \calA$.
Also, the centralizer of $\calA \otimes_\K \L$ in $\Mat_n(\K) \otimes_\K \L$
is clearly $\calC(\calA) \otimes_\K \L$, therefore
$\calA$ has a trivial centralizer in $\Mat_n(\K)$ if and only if
$\calA_\L$ has a trivial centralizer in $\Mat_n(\L)$.
We deduce that
$$t_n(\K) \geq t_n(\L).$$
Therefore, by Steinitz's theorem and the examples discussed earlier, it will suffice to prove theorem \ref{maintheo} when $\K$ is algebraically closed.

\vskip 3mm
\begin{center}
In the rest of this section, we assume $\K$ is algebraically closed.
\end{center}

\subsection{The case $\calA$ is unispectral}\label{unispect}

\begin{Def}
We call a matrix $A\in \Mat_n(\K)$ \textbf{unispectral} when it has a sole eigenvalue. \\
A subalgebra of $\Mat_n(\K)$ is called unispectral when all its elements are unispectral.
\end{Def}

The standard example is the subalgebra of matrices of the form $\lambda.I_n+T$ for some strictly upper triangular matrix $T$.
Conversely, every unispectral subalgebra is conjugate to a subalgebra of the preceding one:

\begin{prop}
Let $\calA$ be a unispectral subalgebra of $\Mat_n(\K)$. Then there is a non-singular $P \in \GL_n(\K)$
such that $P\calA P^{-1} \subset T_n(\K)$.
\end{prop}

\begin{proof}
We use Burnside's theorem (see \cite{Jac2} p.213) by induction on $n$.
The case $n=1$ is trivial.
Assume $n \geq 2$ and our claim holds for every non-negative integer $p<n$ and every unispectral subalgebra of $\Mat_p(\K)$.
Let $\calA$ be a unispectral subalgebra of $\Mat_n(\K)$. Clearly, $\calA \subsetneq \Mat_n(\K)$
hence Burnside's theorem shows there is a non-singular $P \in \GL_n(\K)$ and an integer $p \in \lcro 1,n-1\rcro$
such that every $M \in \calA$ splits as
$$M=P^{-1}\,\begin{bmatrix}
A(M) & * \\
0 & B(M)
\end{bmatrix}P \quad \text{where $A(M) \in \Mat_p(\K)$ and $B(M) \in \Mat_{n-p}(\K)$.}$$
Then $\calA_1:=\bigl\{A(M) \mid M \in \calA\bigr\}$ (resp.\ $\calA_2:=\bigl\{B(M) \mid M \in \calA\bigr\}$)
is a unispectral subalgebra of $\Mat_p(\K)$ (resp.\ of $\Mat_{n-p}(\K)$).
Using the induction hypothesis, there are non-singular matrices $P_1 \in \GL_p(\K)$ and $P_2 \in \GL_{n-p}(\K)$ such that
$P_1\calA_1 P_1^{-1} \subset T_p(\K)$ and $P_2\calA_2 P_2^{-1} \subset T_{n-p}(\K)$.
Setting $Q:=\begin{bmatrix}
P_1 & 0 \\
0 & P_2
\end{bmatrix}$, we have $Q \in \GL_n(\K)$ and $(QP)\calA (QP)^{-1} \subset T_n(\K)$.
\end{proof}

\begin{cor}
Let $p \geq 2$ be an integer.
Then every unispectral subalgebra of $\Mat_p(\K)$ has a non-trivial centralizer.
\end{cor}

\begin{proof}
It suffices to prove the statement for any unispectral subalgebra $\calA$ of $T_p(\K)$.
However any matrix $M$ of $\calA$ must have identical diagonal entries and must be upper triangular:
it easily follows that the elementary matrix $E_{1,n}$ lies in $\calC(\calA)$.
\end{proof}

\noindent
In what follows, we will assume $\calA$ is not unispectral.

\subsection{The basic splitting}\label{split}

Let us choose a non-unispectral matrix $M \in \calA$.
Choose then a spectral projection $P$ associated to $M$: hence $P \in \K[M] \subset \calA$
(see \cite{Chamb} chapter 8, §4 corollary 3 p.271) and we have therefore found a non-trivial idempotent in $\calA$.
By conjugating $\calA$ with an appropriate non-singular matrix, we are reduced to the case $\calA$
contains, for some $p \in \lcro 1,n-1\rcro$, the idempotent matrix
$$P:=\begin{bmatrix}
I_p & 0 \\
0 & 0
\end{bmatrix} \in \Mat_n(\K).$$

\begin{Rem}
The rest of the proof will only rely on this assumption and the condition on the dimension of $\calA$.
In particular, the reader will check that it does not use the fact that $\K$ is algebraically closed.
\end{Rem}

\noindent In what follows, we set $q:=n-p$. Notice then that $\calA$ also contains $Q:=I_n-P=\begin{bmatrix}
0 & 0 \\
0 & I_q
\end{bmatrix}$.
Since $\calA$ is a subalgebra containing both $P$ and $Q$, one clearly has:
$$P\,\calA\,P=\Biggl\{
\begin{bmatrix}
M & 0 \\
0 & 0
\end{bmatrix}
\mid M \in \Mat_p(\K)\Biggr\} \,\cap \, \calA \quad, \quad
P\,\calA\,Q=\Biggl\{
\begin{bmatrix}
0 & M \\
0 & 0
\end{bmatrix}
\mid M \in \Mat_{p,q}(\K)\Biggr\} \,\cap \, \calA,$$
$$Q\,\calA\,P=\Biggl\{
\begin{bmatrix}
0 & 0 \\
M & 0
\end{bmatrix}
\mid M \in \Mat_{q,p}(\K)\Biggr\} \,\cap \, \calA, \quad \textrm{and} \quad
Q\,\calA\,Q=\Biggl\{
\begin{bmatrix}
0 & 0 \\
0 & M
\end{bmatrix}
\mid M \in \Mat_q(\K)\Biggr\} \,\cap \, \calA.$$
Those four linear subspaces of $\calA$ will respectively be denoted by
$\calA_{1,1}$, $\calA_{1,2}$, $\calA_{2,1}$ and $\calA_{2,2}$.
For every $M \in \Mat_n(\K)$, write:
$$M=I_n\,M\,I_n=(P+Q)\,M\,(P+Q)
=P\,M\,P+P\,M\,Q+Q\,M\,P+Q\,M\,Q.$$
It follows that
$$\calA=\calA_{1,1} \oplus \calA_{1,2} \oplus \calA_{2,1} \oplus \calA_{2,2}$$
hence
$$\dim \calA=\dim \calA_{1,1}+\dim \calA_{1,2}+\dim \calA_{2,1}+\dim \calA_{2,2.}$$
Notice also that $P \in \calA_{1,1}$ and $Q \in \calA_{2,2}$, so that
$$\dim \calA_{1,1} \geq 1 \quad \textrm{and} \quad \dim \calA_{2,2} \geq 1.$$
In what follows, we will discuss various cases depending on the respective dimensions of the $\calA_{i,j}$'s.
One case can readily be done away:
if $\calA_{1,2}=\{0\}$ and $\calA_{2,1}=\{0\}$, then $P$ belongs to $\calC(\calA) \setminus \Vect(I_n)$.

\begin{center}
We will now assume $\calA_{1,2} \neq \{0\}$ or $\calA_{2,1}\neq \{0\}$.
\end{center}

\subsection{The case $\dim \calA=3$}\label{dim3}

It only remains to investigate the case one of $\calA_{2,1}$ and $\calA_{1,2}$
has dimension $1$ and the other $0$. Transposition of $\calA$ only leaves us with the case
$\calA_{2,1}=\{0\}$ and $\calA_{1,2}$ is generated by some $C \in \Mat_{p,q}(\K) \setminus \{0\}$.
Hence $\calA_{1,1}=\Vect(I_p)$ and $\calA_{2,2}=\Vect(I_q)$, so the
considerations of Section \ref{proofexamples} show that, in order to prove that $\calC(\calA)$ is non-trivial,
it will suffice to prove that, for some pair  $(X,Y) \in \Mat_p(\K) \times \Mat_q(\K)$, one has $XC=CY$ and 
$\begin{bmatrix}
X & 0 \\
0 & Y
\end{bmatrix}$ is not a scalar multiple of $I_n$.
Consider then the linear map:
$$f : \begin{cases}
\Mat_p(\K) \times\Mat_q(\K) & \longrightarrow \Mat_{p,q}(\K) \\
(X,Y) & \longmapsto XC-CY.
\end{cases}$$
By the rank theorem:
$$\dim \Ker f \geq \dim \bigl(\Mat_p(\K) \times\Mat_q(\K)\bigr)-\dim \Mat_{p,q}(\K)
=p^2+q^2-p\,q=(p-q)^2+p\,q \geq p\,q.$$
Since $n \geq 3$ and $p \in \lcro 1,n-1\rcro$, one has $p\,q \geq 2$,
which shows that $\Ker f \neq \Vect((I_p,I_q))$. Therefore, $\calC(\calA)$ is non-trivial.

\subsection{The case $\dim \calA=4$}\label{dim4}

We now assume $\dim \calA=4$ and set
$$\nu(\calA)=\bigl(\dim \calA_{1,1},\dim\calA_{1,2}, \dim \calA_{2,1},\dim \calA_{2,2}\bigr).$$
Transposition and conjugation by a permutation matrix help us reduce the situation to only three cases:
\begin{itemize}
\item $\nu(\calA)=(2,1,0,1)$;
\item $\nu(\calA)=(1,1,1,1)$;
\item $\nu(\calA)=(1,2,0,1)$.
\end{itemize}
In the first two cases, we will show that $\calC(\calA)$ is non-trivial, even when $n$ is odd.
In the last case, we will show that $\calC(\calA)$ is non-trivial when $n$ is even.

\subsubsection{The case $\nu(\calA)=(2,1,0,1)$.}
In this case, there is some $C \in \Mat_p(\K) \setminus \Vect(I_p)$ and some $V \in \Mat_{p,q}(\K) \setminus \{0\}$ such that
$\calA$ is generated by the matrices
$$\begin{bmatrix}
I_p & 0 \\
0 & 0
\end{bmatrix}, \; \begin{bmatrix}
0 & 0 \\
0 & I_q
\end{bmatrix}, \; \begin{bmatrix}
C & 0 \\
0 & 0
\end{bmatrix} \; \textrm{and} \;
\begin{bmatrix}
0 & V \\
0 & 0
\end{bmatrix}.$$
Since the product of the last two matrices belongs to $\calA$,
one must have $C\,V=\lambda\,V$ for some $\lambda \in \K$. Replacing $C$ with $C-\lambda.I_p$,
we may assume $C\,V=0$, in which case the matrix
$\begin{bmatrix}
C & 0 \\
0 & 0
\end{bmatrix}$ clearly belongs to $\calC(\calA) \setminus \Vect(I_n)$.

\subsubsection{The case $\nu(\calA)=(1,1,1,1)$.}
In this case, there are matrices $U\in \Mat_{q,p}(\K) \setminus \{0\}$ and
$V \in \Mat_{p,q}(\K) \setminus \{0\}$ such that $\calA$ is generated by the four matrices
$$\begin{bmatrix}
I_p & 0 \\
0 & 0
\end{bmatrix}, \; \begin{bmatrix}
0 & 0 \\
0 & I_q
\end{bmatrix}, \; \begin{bmatrix}
0 & 0 \\
U & 0
\end{bmatrix} \; \textrm{and} \;
\begin{bmatrix}
0 & V \\
0 & 0
\end{bmatrix}.$$
A matrix belongs to $\calC(\calA)$ if and only if it has the form
$\begin{bmatrix}
X & 0 \\
0 & Y
\end{bmatrix}$, where $(X,Y) \in \Mat_p(\K) \times \Mat_q(\K)$ satisfies
$UX=YU$ and $XV=VY$. \\
One obvious element in this centralizer is
$\begin{bmatrix}
VU & 0 \\
0 & UV
\end{bmatrix}$.
Assume now that $\calC(\calA)$ is trivial.
There would then exist some $\lambda \in \K$ such that $VU=\lambda.I_p$ and $UV=\lambda.I_q$.
Two situations may arise:
\begin{itemize}
\item The case $\lambda \neq 0$.
Replacing $V$ by $\frac{1}{\lambda}\,V$, we may assume $\lambda=1$.
Then standard rank considerations show that $p=q$ and $V=U^{-1}$. Equality
$XV=VY$ thus implies $UX=YU$, hence $\calA$ has the same centralizer as the subalgebra generated by
$P$, $Q$ and $\begin{bmatrix}
0 & V \\
0 & 0
\end{bmatrix}$, which has been shown to be non-trivial in Section \ref{dim3}. There lies a contradiction.

\vskip 2mm
\item The case $\lambda=0$. Then $UV=0$ and $VU=0$.
Since neither $U$ nor $V$ is zero, we deduce that
$\Ker U$ and $\im V$ are non-trivial subspaces of $\K^p$ and
$\Ker V$ and $\im U$ are non-trivial subspaces of $\K^q$.
We can therefore construct rank $1$ matrices $X \in \Mat_p(\K)$ and $Y \in \Mat_q(\K)$ such that
$\im X \subset \Ker U$, $\im V \subset \Ker X$, $\im Y \subset \Ker V$ and
$\im U \subset \Ker Y$. Hence $XV=VY=0$ and $UX=YU=0$, and it follows that the block diagonal matrix
$\begin{bmatrix}
X & 0 \\
0 & Y
\end{bmatrix}$ belongs to $\calC(\calA) \setminus \Vect(I_n)$, another contradiction.
\end{itemize}
The previous \emph{reductio ad absurdum} then proves that $\calC(\calA)$ is non-trivial.

\subsubsection{The case $\nu(\calA)=(1,2,0,1)$.}
We actually lose no generality assuming $p \leq q$ (if not, we simply replace $\calA$ with $\calA^t$
and conjugate it by an appropriate permutation matrix).
We then find linearly independent matrices $U$ and $V$ in $\Mat_{p,q}(\K)$ such that
$\calA$ is generated by
$$\begin{bmatrix}
I_p & 0 \\
0 & 0
\end{bmatrix}, \; \begin{bmatrix}
0 & 0 \\
0 & I_q
\end{bmatrix}, \; \begin{bmatrix}
0 & U \\
0 & 0
\end{bmatrix} \; \textrm{and} \;
\begin{bmatrix}
0 & V \\
0 & 0
\end{bmatrix}.$$
The considerations of Section \ref{proofexamples}
show that the kernel of
$$g : \begin{cases}
\Mat_p(\K) \times \Mat_q(\K) & \longrightarrow \Mat_{p,q}(\K)^2 \\
(X,Y) & \longmapsto (XU-UY\,,\,XV-VY)
\end{cases}$$
has a dimension greater than $1$ if and only if $\calC(\calA)$ is non-trivial.

\begin{lemme}\label{applilin2}
With the preceding notations, for to have $\dim \Ker g=1$, it is necessary that:
\begin{itemize}
\item either $q=p$ and at least one of the matrices $U$ or $V$ has rank $p$;
\item or $q=p+1$ and both matrices $U$ and $V$ have rank $p$.
\end{itemize}
\end{lemme}

In order to prove this lemma, we will start with another one:

\begin{lemme}\label{notontolemma}
Let $W \in \Mat_{p,q}(\K)$ be such that $\rk W<p \leq q$. Then the linear map
$$f : \begin{cases}
\Mat_p(\K) \times \Mat_q(\K) & \longrightarrow \Mat_{p,q}(\K) \\
(X,Y) & \longmapsto XW-WY
\end{cases}$$
is not onto.
\end{lemme}

\begin{proof}
Notice that $\rk f$ is unchanged by replacing $W$ with an equivalent matrix:
for every $(P_1,Q_1)\in \GL_p(\K) \times \GL_q(\K)$ and
$(X,Y)\in \Mat_p(\K) \times \Mat_q(\K)$, one can indeed write:
$$XP_1WQ_1-P_1WQ_1Y=P_1\,\Bigl[(P_1^{-1}XP_1)W-W(Q_1YQ_1^{-1})\Bigr]\,Q_1.$$
Letting $r:=\rk W$, we then lose no generality assuming that $W=\begin{bmatrix}
I_r & 0 \\
0 & 0
\end{bmatrix}$. In this case however, every matrix of $\im f$
has a zero entry in position $(p,q)$, hence $f$ is not onto.
\end{proof}

\begin{proof}[Proof of Lemma \ref{applilin2}]
setting $f_1 : (X,Y) \longmapsto XU-UY$ and $f_2 : (X,Y) \longmapsto XV-VY$,
we find that $\rk g \leq \rk f_1+\rk f_2$ whilst Lemma \ref{notontolemma} shows that
$\rk f_1+\rk f_2 \leq 2\,\Mat_{p,q}(\K)-m$ where $m$ is the number of matrices in $\{U,V\}$
which have rank lesser than $p$. We deduce that $\rk g \leq 2\,p\,q-m$.
The rank theorem then shows that
$$\dim \Ker g \geq p^2+q^2-2\,p\,q+m=(q-p)^2+m.$$
If $q \geq p+2$, then $\dim \Ker g \geq 4$. If
$q=p+1$ and $m \geq 1$, then $\dim \Ker g \geq 2$.
If $q=p$ and $m=2$, then $\dim \Ker g \geq 2$. This proves the claimed results. 
\end{proof}

\vskip 4mm
\noindent We may now complete the proof of Theorem \ref{maintheo}.
Assume $n$ is even. Then Lemma \ref{applilin2}
shows $\calC(\calA)$ is non-trivial unless $p=q=\frac{n}{2}$ and one of the matrices $U$ and $V$ is non-singular. \\
Assume then $p=q=\frac{n}{2}$ and $U$ is non-singular (for example). 
Conjugating $\calA$ by
$\begin{bmatrix}
I_p & 0 \\
0 & U
\end{bmatrix}$, we are then reduced to the case $U=I_p$: in this case $V \not\in \Vect(I_p)$ and
the matrix $\begin{bmatrix}
V & 0 \\
0 & V
\end{bmatrix}$ clearly belongs to $\calC(\calA) \setminus \Vect(I_n)$.

\vskip 4mm
\noindent Let us finish this section by summing up the results in the case $\dim \calA=4$.
Recall that we have assumed that $\K$ is algebraically closed or that $\calA$ contains a non-trivial idempotent. 
\begin{enumerate}[(i)]
\item If $n$ is even, then $\calC(\calA)$ is non-trivial.
\item If $n$ is odd and $\calC(\calA)$ is trivial, then, setting $p:=\frac{n-1}{2}$,
there are linearly independent matrices $U$ and $V$ in $\Mat_{p,p+1}(\K)$ with rank $p$
and a non-singular matrix $P \in \GL_n(\K)$ such that
either
$$P\,\calA\,P^{-1}=\Biggl\{\begin{bmatrix}
a.I_p & c.U+d.V \\
0 & b.I_{p+1}
\end{bmatrix} \mid (a,b,c,d)\in \K^4\Biggr\}$$
or
$$P\,\calA^t\,P^{-1}=\Biggl\{\begin{bmatrix}
a.I_p & c.U+d.V \\
0 & b.I_{p+1}
\end{bmatrix} \mid (a,b,c,d)\in \K^4\Biggr\}.$$
\end{enumerate}
This of course completes the proof of Theorem \ref{maintheo}.

\section{On subalgebras of dimension $4$ of $\Mat_{2n+1}(\K)$ with a trivial centralizer}

Here, we establish Proposition \ref{unicity} by prolonging the proof of Section \ref{dim4} in the case
$\dim \calA=4$. We must first return to the situation where $\K$ is an arbitrary field.

\begin{lemme}\label{ranklemma}
Let $n \in \N^*$ and $\calA$ be a $4$-dimensional subalgebra of $\Mat_{2n+1}(\K)$
with a trivial centralizer. Then $\calA$ contains a rank $n$ idempotent.
\end{lemme}

\begin{proof}
Choose an algebraically closed extension $\L$ of $\K$.
Then $\calA_\L$ has a trivial centralizer in $\Mat_{2n+1}(\L)$.
The proof in Sections \ref{red}, \ref{unispect}, \ref{split} and \ref{dim4}
then shows that there is a $2$-dimensional subspace $P \subset \Mat_{n,n+1}(\K)$
such that $\calA_\L$ is conjugate to either the subalgebra
$$\calH:=\Biggl\{\begin{bmatrix}
a.I_n & M \\
0 & b.I_{n+1}
\end{bmatrix} \mid (a,b)\in \L^2, \; M \in P\Biggr\}$$
or its transpose $\calH^t$.
In any case, the set of unispectral elements in $\calA_\L$
is a linear hyperplane of $\calA_\L$:
this is the case indeed when $\calA_\L=\calH$ since this subset is then
$$\Biggl\{\begin{bmatrix}
a.I_n & M \\
0 & a.I_{n+1}
\end{bmatrix} \mid a\in \L, \; M \in P\Biggr\}.$$
Also, every non-unispectral element of $\calA_\L$
is clearly diagonalisable with exactly two eigenvalues of respective orders $n$ and $n+1$. \\
Every basis of the $\K$-vector space $\calA$ is also a basis of the $\L$-vector space
$\calA_\L$ and therefore must contain a matrix which is not unispectral in $\Mat_n(\L)$.
Let us choose such a matrix $M \in \calA$, with eigenvalues $\lambda$ and $\mu$ of respective orders $n$ and $n+1$.
Notice then that $\lambda$ and $\mu$ belong to $\K$. Indeed:
\begin{itemize}
\item the minimal polynomial of $M$ is
$X^2-(\lambda+\mu)X+\lambda\,\mu$, so $\lambda+\mu \in \K$ (implicit here is the fact that the minimal polynomial of a matrix is
unchanged by extending the field of scalars);
\item also $\tr(M)=n(\lambda+\mu)+\mu$, which entails $\mu \in \K$ and therefore $\lambda \in \K$.
\end{itemize}
We deduce that $\frac{1}{\lambda-\mu}\,(M-\mu.I_{2n+1})$ is a rank $n$ idempotent in $\calA$.
\end{proof}

\noindent Proposition \ref{unicity} can now be proven.
Since $\calA$ contains an idempotent of rank $n$,
the proof from Sections \ref{split} and \ref{dim4} shows that we can reduce the study to the situation where there
is a $2$-dimensional subspace $\calP \subset \Mat_{n,n+1}(\K)$ such that
$$\calA=\Biggl\{\begin{bmatrix}
a.I_n & M \\
0 & b.I_{n+1}
\end{bmatrix} \mid (a,b)\in \K^2, \; M \in \calP\Biggr\}.$$
Let $U \in \calP \setminus \{0\}$ and choose $V$ such that $(U,V)$ is a basis of $\calP$.
Then Lemma \ref{ranklemma} shows that $U$ (and $V$) must have rank $n$.
For every extension $\L$ of $\K$, the subalgebra $\calA_\L$ has a trivial centralizer, which shows
$\rk U=n$ for every $U \in P_\L \setminus \{0\}$. We will then use the following result to see that $\calP$
is equivalent to the $2$-dimensional subspace $\Vect(C_n,D_n)$ i.e.\
$\calP=P\,\Vect(C_n,D_n)\,Q$ for some pair $(P,Q)\in \GL_n(\K) \times \GL_{n+1}(\K)$:

\begin{prop}\label{kronecker}
Let $A$ and $B$ in $\Mat_{n,n+1}(\K)$.
Assume that every non-trivial linear combination of $A$ and $B$ over any field extension of $\K$
has rank $n$. Then there is a pair $(P,Q)\in \GL_n(\K) \times \GL_{n+1}(\K)$
such that $A=P\,C_n\,Q^{-1}$ and $B=P\,D_n\,Q^{-1}$.
\end{prop}

\noindent Assuming this to be true, consider a basis $(A,B)$ of $\calP$ and a pair $(P,Q)$
associated to it as in Proposition \ref{kronecker}.
Then a straightforward computation shows that $R\,\calH_{2n+1}(\K)\,R^{-1}=\calA$
for $R=\begin{bmatrix}
P & 0 \\
0 & Q
\end{bmatrix}$, which completes the proof of Theorem \ref{unicity}.

\begin{proof}[Proof of Proposition \ref{kronecker}]
We let $x$ denote an indeterminate.
We use the Kronecker-Weierstrass reduction for pencils of matrices (see chapter XII of \cite{Gantmacher} and
the appendix of \cite{dSPsim} for the case of an arbitrary field).
Since $A$ and $B$ have rank $n$ and belong to $\Mat_{n,n+1}(\K)$,
the canonical form of the pencil $A+x\,B$ cannot contain any block of the forms
$$\begin{bmatrix}
x & 0 &  & & \\
1 & x & 0 & &  \\
0 & 1 & x & &  \\
 & & \ddots & \ddots &  \\
 & & & \ddots & x \\
 &        &  & & 1
\end{bmatrix}\in \Mat_{p+1,p}(\K[x]) \quad ; \quad
\begin{bmatrix}
1 & x & 0 &  \\
0 & 1 & x &  \\
 & & \ddots & \ddots   \\
  & & & \ddots & x \\
 &        &  & & 1
\end{bmatrix}\in \Mat_p(\K[x])$$
nor
$$\begin{bmatrix}
x & 1 & 0 &  \\
0 & x & 1 &  \\
 & & \ddots & \ddots   \\
   & & & \ddots & 1 \\
 &        &  & & x
\end{bmatrix}\in \Mat_p(\K[x])$$
and contains at most one block of the form
$$L_p:=\begin{bmatrix}
1 & x & 0 & &  \\
0 & 1 & x & &  \\
 & & \ddots & \ddots &  \\
 &        &  & 1 & x
\end{bmatrix} \in \Mat_{p,p+1}(\K[x]).$$
It follows that there exists some $p \in \lcro 0,n-1\rcro$, some non-singular $C \in \GL_p(\K)$
and some pair $(P,Q)\in \GL_n(\K) \times \GL_{n+1}(\K)$ such that
$$P^{-1}\,(A+x\,B)\,Q=
\begin{bmatrix}
C+x.I_p & 0 \\
0 & L_{n-p}
\end{bmatrix}.$$
However, if $p>0$, then $\rk(A-\lambda\,B)<n$ for any eigenvalue $\lambda$ of $C$,
which contradicts the assumptions. It follows that $p=0$, hence
$P^{-1}\,(A+x\,B)\,Q=L_n=C_n+x\,D_n$, which shows $A=P\,C_n\,Q^{-1}$ and $B=P\,D_n\,Q^{-1}$.
\end{proof}

\section*{Acknowledgements}

I would like to thank Saab Abou-Jaoud\'e for submitting me the original problem
and Pierre Mazet for inspiring part of this work.

\end{document}